\documentclass[12pt]{article}
\usepackage{amssymb,amsmath,amsthm}
\usepackage{amsmath}
\usepackage{setspace}

\newtheorem{theorem}{Theorem}[section]
\newtheorem{corollary}[theorem]{Corollary}
\newtheorem{lemma}[theorem]{Lemma}
\newtheorem{conjecture}[theorem]{Conjecture}
\newtheorem{definition}[theorem]{Definition}
\newtheorem{example}[theorem]{Example}

\newcommand{\R}{\mathbb R}
\newcommand{\C}{\mathbb C}
\newcommand{\Z}{\mathbb Z}
\newcommand{\CB}{\mathcal B}

\newcommand{\CC}{\mathcal C}
\newcommand{\CP}{\mathcal P}
\newcommand{\CQ}{\mathcal Q}

\begin{document}
\title{Distinct values of bilinear forms on algebraic curves}
\author{Claudiu Valculescu
  \and 
    Frank de Zeeuw
    }
\date{}
\maketitle

\vspace{-20pt}
\begin{abstract}
Let $B_M:\C^2\times\C^2\to\C$ be a bilinear form $B_M(p,q)=p^TMq$, with an invertible matrix $M\in \C^{2\times2}$.
 We prove that any finite set $S$ contained in an irreducible algebraic curve $C$ of degree $d$ in $\C^2$
 determines $\Omega_d(|S|^{4/3})$ distinct values of $B_M$,
 unless $C$ has an exceptional form.
This strengthens a result of Charalambides \cite{Ch13} in several ways. 
  
 The proof is based on that of Pach and De Zeeuw \cite {PZ14}, who proved a similar statement for the Euclidean distance function in $\R^2$.
Our main motivation for this paper is that for bilinear forms,
 this approach becomes more natural, and should better lend itself to understanding and generalization. 
\end{abstract}

\section{Introduction}\label{sec:intro}
Pach and De Zeeuw \cite{PZ14} proved that a finite set $S$ on an irreducible algebraic curve of degree $d$ in $\R^2$ determines $\Omega_d(|S|^{4/3})$ distinct Euclidean distances,
unless that curve is a line or a circle. 
In this paper we prove an analogous result for functions $\C^2\times \C^2\to \C$ of the following form, with $p=(p_x,p_y),q=(q_x,q_y)\in \C^2$:
$$
  c_1p_xq_x + c_2p_xq_y
+ c_3p_yq_x + c_4p_yq_y.
$$
We refer to such functions as \emph{bilinear forms},
and write them more compactly as 
$$B_M(p,q) := p^TMq$$
with a matrix $M\in \C^{2\times 2}$.
We assume throughout that $M$ is invertible.
For $S\subset\C^2$, we write
$\CB_M(S):=\{B_M(p,q):p,q\in S\}$,
so $|\CB_M(S)|$ is the number of distinct values of $B_M$ on $S$.

Two particular functions that we are interested in are $B_I$ and $B_A$ for  
$$I :=\begin{pmatrix}1&0\\0&1\end{pmatrix},~~~~~ 
A :=\begin{pmatrix}0&1\\-1&0\end{pmatrix}.$$
Over $\R$, $B_I(p,q) = p^Tq$ is the dot product, and $B_A(p,q)$ is twice the signed area of the triangle spanned by $p$, $q$, and the origin.
Distinct values of the dot product on various sets were considered in \cite{St10} and \cite[Chapter 9]{GIS11},
but have not been considered on algebraic curves before.
For triangle areas, Charalambides \cite{Ch13} proved (among other results) that for $S$ contained in an algebraic curve of degree $d$ in $\R^2$,
one has $|\CB_A(S)| = \Omega_d(|S|^{5/4})$, unless the curve is a line, an ellipse centered at the origin, or a hyperbola centered at the origin. 
We improve Charalambides's bound to $\Omega_d(|S|^{4/3})$, give an explicit dependence on $d$, 
and extend our bound to general bilinear forms as well as to curves in $\C^2$. 

The class of curves for which our bound does not hold is actually somewhat larger than for Charalambides, so, strictly speaking, we do not quite improve his bound in all cases.
But we show that our class of exceptional curves is best possible for general bilinear forms.
This class is captured in the following definition.

\begin{definition}\label{def:special}
We call an algebraic curve in $\C^2$ a \emph{special curve} if it is a line, or it is linearly equivalent to a curve defined by an equation of the form
$$x^k=y^\ell,~~~\text{with}~~~k,\ell\in \Z\backslash\{0\},~~~ \gcd(k,\ell)=1.$$
\end{definition}

We say that two curves $C,C'$ are \emph{linearly equivalent} if there is an invertible matrix $D\in \C^{2\times 2}$ such that $C' =DC :=\{Dp:p\in C\}$.
Because $k$ and $\ell$ are assumed to be coprime, all special curves are irreducible.
When $k$ or $\ell$ is negative, one obtains a more natural polynomial equation after multiplying by an appropriate monomial.
Thus special curves include hyperbola-like curves of the form $x^ky^\ell=1$ with coprime $k,\ell\geq 1$. Ellipses centered at the origin are also included, since these are linearly equivalent to the unit circle $(x-iy)(x+iy)=1$, which is linearly equivalent to $xy=1$. Thus all the exceptional curves of Charalambides are special.

We now show that for any special curve, there is a bilinear form that takes only a linear number of distinct values on it. 

\begin{example}
If $C$ is special, there are $M\in \C^{2\times 2}$ and $S\subset C$ such that $|B_M(S)|=O(|S|)$.
\begin{itemize}
\item Let $C$ be a line $y=c$.
Then $S=\{(2^i,c):i=1,\ldots,|S|\}$ has $|\CB_I(S)| =O(|S|)$.
\item Consider the curve $C$ given by $x^k = y^\ell$. Take
$$S := \{(2^{\ell i}, 2^{k i}): i=1,\ldots, |S|\}\subset C.$$
Then $B_I\left((2^{\ell i},2^{ki}),(2^{\ell j},2^{kj})\right)
= (2^\ell)^{i+j}+(2^k)^{i+j}$,
so $|\CB_I(S)| =O(|S|)$. 
\item 
For any other special curve $C'$, there is an invertible matrix $D$ such that $C' = DC$, for a curve $C$ defined by $x^k=y^\ell$ or $y=c$. 
Then we can choose $S\subset C$ as above, so that for $p,q\in S$, we have 
$$B_M(Dp,Dq) = p^T D^TMDq.$$
Choosing $S' = DS \subset C'$ and $M=D^{-T}D^{-1}$, we have $|\CB_M(S')| = |\CB_I(S)| = O(|S'|)$.
\end{itemize}
\end{example}

Our main theorem says that these special curves are the only curves on which $B_M$ could have a linear number of distinct values, while on any other curve $B_M$ must take significantly more values. 
See Section \ref{sec:discussion} for a discussion of extensions and generalizations.

\begin{theorem}\label{thm:main}
Let $C$ be an irreducible algebraic curve in $\C^2$ of degree $d$,
 $S\subset C$ a finite set,
 and $B_M$ a bilinear form as above.
If $C$ is not special, then $$|\CB_M(S)|=\Omega\left(d^{-14/3}|S|^{4/3}\right).$$
\end{theorem}
\begin{proof}
We outline how the proof is distributed over the paper.
By Corollary \ref{cor:main2} in Section \ref{sec:proof}, the bound holds if $M$ is invertible and $C$ has $O(d^2)$ automorphisms.
By Theorem \ref{thm:autos} in Section \ref{sec:autos}, the only curves that do not have $O(d^2)$ automorphisms are the special curves.
\end{proof}

For clarity we have chosen not to state our result in the most general form possible.
The proof in fact gives a ``bipartite'' statement (see Theorem \ref{thm:main2} and also \cite{PZ14}), and can be extended to bilinear functions with linear terms, as well as to reducible curves.
We also note that for sets on curves in $\R^2$, our proof gives, with a little extra work, a better dependence on $d$, namely $d^{-2}$ instead of $d^{-14/3}$.

Our proof follows the setup in \cite{PZ14}, which is based on that of \cite{SSS13}.
It turns out that, for bilinear forms, this setup leads to a more natural and streamlined proof than for the Euclidean distance function in \cite{PZ14}.
This was our main motivation for working out this variant in detail,
and we hope that it helps to clarify the proof of \cite{PZ14},
and increases the potential for generalization. 
We also wanted to test the limits of this approach, by extending it to complex curves and by explicitly determining the dependence on the degree of the curve.
In future work we hope to study more general polynomial functions, as well as functions on curves in higher dimensions.

Let us quickly give the relevant definitions.
A set $C\subset \C^2$ is an \textit{algebraic curve} if there is an $f\in \C[x,y]\backslash\{0\}$ such that
$C =Z(f):=\{ (x,y)\in \C^2: f(x,y) = 0 \}$.
The {\it degree} of $C$ is the minimum degree of a polynomial $f$
such that $C = Z(f)$.
The curve $C$ is \emph{irreducible} if there is an irreducible $f$ such that $C=Z(f)$.
We frequently use \emph{B\'ezout's inequality}, which states that the number of intersection points of two distinct irreducible algebraic curves in $\C^2$ is at most the product of their degrees.
In our proof, we also consider algebraic curves in $\C^4$; for their definition, we refer to \cite{Ha77}.
A crucial role in the proof is played by linear automorphisms of curves.
A \emph{linear automorphism} of an algebraic curve $C$ is an invertible linear transformation $T:\C^2\to \C^2$ such that $T(C) = C$.
We often drop the word ``linear''.



\section{Proof of Theorem \ref{thm:main}}\label{sec:proof}

In this section we give one side of the proof of Theorem \ref{thm:main}; the other side follows in Section \ref{sec:autos}.
We prove Theorem \ref{thm:main2}, a variant of Theorem \ref{thm:main} that is more convenient for the proof, and deduce Corollary \ref{cor:main2}, which, together with Theorem \ref{thm:autos}, directly implies Theorem \ref{thm:main}.

\subsection{A variant of Theorem \ref{thm:main}}

Theorem \ref{thm:main2} differs from Theorem \ref{thm:main} in the following ways. It focuses on the matrix $I$ (i.e., $B_I(p,q)=p^Tq$ is the ``dot product''), but the statement is slightly more general, in that it bounds the values of the function in a useful ``bipartite'' way;
for $S_1,S_2\subset \C^2$, it bounds the size of  
$\CB_I(S_1,S_2) := \{B_I(p,q): p\in S_1, q\in S_2\}$.
This more general form allows us to deduce the result for $B_M$.
Finally, the exceptional curves in Theorem \ref{thm:main2} are those curves that have many automorphisms. In Section \ref{sec:autos}, we show that the only curves with many automorphisms are the special curves of Definition \ref{def:special}.

\begin{theorem}\label{thm:main2}
 Let $C_1$ and $C_2$ be irreducible algebraic curves in $\C^2$, both of degree at most $d$,
 and let $S_1\subset C_1, S_2\subset C_2$ be disjoint finite sets with $|S_1|=|S_2|=n$.
If $C_1$ and $C_2$ each have $O(d^2)$ automorphisms, then
 $$|\CB_I(S_1,S_2)|=\Omega\left(d^{-14/3}n^{4/3}\right).$$
\end{theorem}

We first deduce from this theorem a statement that is closer to Theorem \ref{thm:main}.

\begin{corollary}\label{cor:main2}
Let $C$ be an irreducible algebraic curve in $\C^2$ of degree $d$,
 $S\subset C$ a finite set,
 and $B_M$ a bilinear form.
 If $M$ is invertible and $C$ has $O(d^2)$ automorphisms, then
$$|\CB_M(S)|=\Omega\left(d^{-14/3}|S|^{4/3}\right).$$
\end{corollary}
\begin{proof}
We arbitrarily split $S$ into two disjoint sets $S_1,S_2'$ of the same size (discarding one point if $|S|$ is odd).
Then we set $S_2 := MS_2'$.
For $p\in S_1, q'\in S_2'$ we have $B_M(p,q') = B_I(p,Mq') = B_I(p,q)$ with $q\in S_2$.
We set $C_1:=C$ and $C_2:=MC$.
Applying Theorem \ref{thm:main2} to $S_1\subset C_1$ and $S_2\subset C_2$ gives
\[|\CB_M(S)|
=\Omega\left(|\CB_M(S_1,S_2')|\right)
=\Omega\left(|\CB_I(S_1,S_2)|\right)
=\Omega\left(d^{-14/3}|S|^{4/3}\right).\qedhere\]
\end{proof}

\subsection{Preparation}

In the rest of Section \ref{sec:proof} we prove Theorem \ref{thm:main2}.
We assume throughout that $C_1$ and $C_2$ have $O(d^2)$ automorphisms, so in particular they are not lines.

The matrices in the following definition play an important role in the proof.

\begin{definition}\label{def:matrixN}
Given two points $p_i=(x_i,y_i),p_k=(x_k,y_k)\in \C^2$, we define the matrix
$$N_{ik} := 
\begin{pmatrix}
x_i&y_i\\x_k&y_k
\end{pmatrix}.$$
\end{definition}
To ensure that these matrices behave nicely, we prepare the sets $S_1,S_2$ as follows.

\begin{lemma}\label{lem:Sstar}
There is $S^*\subset S_1$ with $|S^*|\geq n/d$ such that any line through the origin contains at most one point of $S^*$.
Consequently, for any distinct $p_i,p_k\in S^*$
the matrix $N_{ik}$ is nonsingular.
Furthermore, there is $T^*\subset S_2$ with the same property and $|T^*|=|S^*|$.
\end{lemma}
\begin{proof}
For any line $L$ through the origin that intersects $S_1$, 
arbitrarily choose one point of $L\cap S_1$ and remove any other point.
Call the result $S^*$.
Since $C_1$ is not a line, 
by B\'ezout it contains at most $d$ points on such a line $L$, 
so $|S^*|\geq n/d$. 

Similarly pick $T^*$ from $S_2$,
and remove points from the larger set until $|S^*|=|T^*|$.
\end{proof}

\textbf{Notation:} The rest of the proof considers only $M=I$, 
so we write $B := B_I$.
We only use the points in $S^*$ and $T^*$; we set $m := |S^*|=|T^*|$ and $\CB = \CB_I(S^*, T^*)$.
Throughout this section we denote points of $C_1$ with the letter $p$, and points of $C_2$ with the letter $q$; 
for points of $S^*$ or $T^*$ we similarly use either $p_i,p_j,\ldots$ or $q_s,q_t,\ldots$ .
As said, we assume throughout that neither $C_1$ nor $C_2$ is a line.


\subsection{Quadruples and curves}
To prove the theorem, we find lower and upper bounds on the number of quadruples in
$$\CQ := \{(p_i,p_j,q_s,q_t):p_i,p_j\in S^*,q_s,q_t\in T^*,\: B(p_i, q_s) = B(p_j, q_t) \}.$$
The lower bound is easily obtained using the Cauchy-Schwarz inequality.

\begin{lemma}\label{lem:lowerbound}
For $\CB$ and $\CQ$ as above we have $ |\CQ| \ge m^4/|\CB|$.
\end{lemma}
\begin{proof}
  Write $B^{-1}(b) := \{(p_i,q_s)\in S^*\times T^*: B(p_i,q_s) = b\}$ for $b\in \CB$.
  Then
\[|\CQ|\geq \sum_{b\in\CB} |B^{-1}(b)|^2
  \geq \frac{1}{|\CB|} \left( \sum_{b\in\CB} |B^{-1}(b)|   \right)^2 = \frac{m^4}{|\CB|}.\qedhere\]
 \end{proof}

To obtain an upper bound on $|\CQ|$, we relate it to an incidence problem for points and curves in $\C^4$.
We define algebraic curves $C_{ij}$ and $\widetilde{C}_{st}$ in $\C^4$ as follows:
For each pair of points $p_i,p_j \in S^*$,
we set
$$C_{ij}:=\{ (q, q') \in \C^4: \: q,q' \in C_2, \: B(p_i,q) = B(p_j, q') \},$$
and for each pair of points $q_s,q_t\in T^*$, we set
$$\widetilde{C}_{st}:=\{ (p, p') \in \C^4: \: p,p' \in C_1, \: B(p,q_s) = B(p', q_t) \}.$$

\begin{lemma}\label{lem:curves}
The sets $C_{ij}$ and $\widetilde{C}_{st}$ are algebraic curves in $\C^4$ of degree at most $d^2$.
\end{lemma}
\begin{proof}
The set $C_{ij}$ is the intersection of the irreducible surface $C_2\times C_2$ and the hyperplane $H_{ij}$ defined by the equation $B(p_i,q) =B(p_j, q')$. 
This hyperplane does not contain the surface, since then fixing $q'$ would give that $C_2$ is a line, which we assumed it is not.
By \cite[Proposition 7.1]{Ha77}, it follows that the intersection is one-dimensional, i.e. it is an algebraic curve.
By a higher-dimensional affine version of B\'ezout's inequality (see \cite[Theorem 7.7]{Ha77} or \cite[Theorem 1]{He83}), the degree of this curve is at most $\deg(C_2)^2\cdot \deg(H_{ij})=d^2$.

The same arguments apply to $\widetilde{C}_{st}$.
\end{proof}

We have $(q_s,q_t)\in C_{ij}$ if and only if $(p_i,p_j)\in \widetilde{C}_{st}$.
This suggests that we can think of the curve $\widetilde{C}_{st}$ as ``dual'' to the point $(q_s,q_t)$,
and of $(p_i,p_j)$ as dual to $C_{ij}$.

Define a point set and a curve set by
$$\CP:= T^*\times T^*,~~~~~~\CC:=\{C_{ij}:(p_i,p_j)\in S^*\times S^*\}.$$
Then a point $(q_s,q_t)\in \CP$ lies on $C_{ij}\in \CC$ if and only if
$(p_i,p_j,q_s,q_t)\in \CQ$. Thus
$$|\CQ| = I(\CP,\CC) := |\{(p,C)\in \CP\times \CC: p\in  C   \}|.$$
It is possible that some $C_{ij}$ coincide as sets, but then we consider them as separate objects.


\subsection{Intersections}
We want to apply an incidence bound to the points $\CP$ and curves $\CC$, and for that we need to control the sizes of the intersections between curves.
We define
$$\CC_{0} := \{C_{ij}\in \CC:
\text{there is a}~C_{kl}\in \CC~\text{such that}~|C_{ij}\cap C_{kl}| = \infty\}$$
and $\CC_1 := \CC\backslash \CC_0$.
Dually, we set
$$\CP_0 := \{(q_s,q_t)\in \CP:
\text{there is a}~(q_u,q_v)\in \CP~\text{such that}~
|\widetilde{C}_{st}\cap\widetilde{C}_{uv}|=\infty\}$$
and $\CP_1 := \CP\backslash \CP_0$.
Thus, the curves in $\CC_0$ are  ``bad'' curves that have large intersection with some other curve, while the points in $\CP_0$ are ``bad'' in a dual sense. 
We show that the sets $\CC_0$ and $\CP_0$ are relatively small.
For the ``good'' sets $\CP_1$ and $\CC_1$, the intersections are well-behaved, allowing us to apply an incidence bound.

With these definitions, two fortunate things happen. Whenever curves $C_{ij}$ coincide as sets, they must lie in $\CC_0$.
The curves $C_{ii}$ for any $i$, which would cause trouble in some of the statements, are also in $\CC_0$, 
because they all contain the line $\{(q,q):q\in C_2\}$.
The analogous statements hold for the dual curves and the corresponding points in $\CP$.

We now show that for $\CP_1$ and $\CC_1$, the intersections are well-behaved.

\begin{lemma}\label{lem:intersections}
For all distinct $C_{ij},C_{kl}\in \CC_1$ we have $$|C_{ij}\cap C_{kl}|\leq d^2,$$
and for any two distinct points in $\CP_1$,
there are at most $d^2$ curves in $\CC$ that contain both.
\end{lemma}
\begin{proof}
As just observed, we can assume that $i\neq j$ and $k\neq l$.
The points $(q,q') \in C_{ij}\cap C_{kl}$ 
are on the intersection of the surface $C_2\times C_2$ with the hyperplanes $H_{ij}:B(p_i,q) = B(p_j,q')$ and $H_{kl}:B(p_k,q) = B(p_l,q')$.
Since, by definition of $\CC_1$, 
$|C_{ij}\cap C_{kl}|$ is finite,
applying B\'ezout's inequality as in Lemma \ref{lem:curves} shows that this intersection contains at most $\deg(C_2)^2\cdot\deg(H_{ij})\cdot \deg(H_{kl}) = d^2$ points.

The same argument gives
$|\widetilde{C}_{st}\cap\widetilde{C}_{uv}|\leq d^2$ for all $s,t,u,v$
with $(q_s,q_t)\neq(q_u,q_v)\in \CP_1$.
This is the dual statement to $(q_s,q_t)$ and $(q_u,q_v)$ lying in at most $d^2$ curves from $\CC$.
\end{proof}

Note that applying B\'ezout's inequality directly to these curves of degree at most $d^2$ gives $|C_{ij}\cap C_{kl}|\leq d^4$, \
which would lead to a worse degree dependence in our final bound.

Next we show that $\CP_0$ and $\CC_0$ are relatively small.
We do this by showing that if two curves have infinite intersection,
then this is related to an automorphism of $C_2$,
and by assumption $C_2$ does not have many automorphisms.

For a linear transformation $T:\C^2\to \C^2$,
we define its \emph{graph on $C_2$} by
$$G_T = \{(q,q')\in\C^4: q,q'\in C_2,~T(q)=q'\}.$$
It is the intersection of the surface $C_2\times C_2$ with the graph of $T$, which is a plane. 
Typically, these two surfaces in $\C^4$ would have finite intersection, but this is not always the case. When the intersection is infinite, this means that $T$ is an automorphism of $C_2$.

\begin{lemma}\label{lem:infiniteintersection}
For any distinct $C_{ij},C_{kl}\in\CC$, there is a linear transformation $T$ such that
$$C_{ij}\cap C_{kl} = G_T.$$
If $|C_{ij}\cap C_{kl}|=\infty$, then   $T$ is an automorphism of $C_2$,
and we have $i\neq k$ and $j\neq l$.

The same statements hold for the dual curves $\widetilde{C}_{st}$
corresponding to points $(q_s,q_t)\in \CP$.
\end{lemma}
\begin{proof}
If $(q,q')\in C_{ij}\cap C_{kl}$ then we have
\begin{align*}
 B(p_i,q) & = B(p_j,q'),\\
 B(p_k,q) & = B(p_l,q'),
\end{align*}
which we can rewrite as $N_{ik}q = N_{jl}q'$
with the matrices $N_{ik},N_{jl}$ from Definition \ref{def:matrixN}.
We have either $i\neq k$ or $j\neq l$; without loss of generality we assume $j\neq l$, so that $N_{jl}$ is invertible by Lemma \ref{lem:Sstar}.
We define a linear transformation $T$ by
$$q' = T(q) = N_{jl}^{-1}N_{ik} q.$$
It follows that $C_{ij}\cap C_{kl}\subset G_T$.
On the other hand, if $(q,q')\in G_T$,
then $q,q' \in C_2$ and $q' = T(q) =N_{jl}^{-1}N_{ik} q$,
so $N_{jl} q'=N_{ik}q$.
This exactly means that $(q,q')\in C_{ij}\cap C_{kl}$,
so in fact we have $C_{ij}\cap C_{kl}= G_T$.
This proves the first statement of the lemma.

 If $|C_{ij}\cap C_{kl}|=\infty$, then $|T(C_2)\cap C_2| =\infty$.
 Since $C_2$ and $T(C_2)$ are irreducible algebraic curves, B\'ezout's inequality implies that $T(C_2) = C_2$,
 i.e., $T$ is an automorphism of $C_2$.

Suppose $i=k$.
If there are infinitely many points $(q,q')\in C_{ij}\cap C_{il}$, 
then they satisfy $N_{ii}q = N_{jl}q'$.
Since $N_{ii}$ is singular and its image is the line $y=x$, the same must be true for $N_{jl}$,
which implies that $j=l$.
Similarly, if $j=l$ and $|C_{ij}\cap C_{kj}|=\infty$, we get $i=k$.

The same arguments give the corresponding statements for the dual curves.
 \end{proof}


\subsection{Incidence bound}

To get an upper bound for the incidences between $\CP_1$ and $\CC_1$,
we use the following theorem, 
which we deduce from a theorem proved by Solymosi and De Zeeuw in \cite{SZ14}.

\begin{theorem}\label{thm:incidencebound}
Let $A,B\subset \C^2$ with $|A|=|B|=\mu$, let $\Pi \subset A\times B$, and let $\Gamma$ be a set of algebraic curves in $\C^4$ of degree at most $\delta$,
with $|\Gamma|=\mu^2$.
If any two points of $\Pi$ are contained in at most $\Delta$ curves of $\Gamma$, then we have
 $$I(\Pi,\Gamma) =O\left(\delta^{4/3}\Delta^{1/3}\mu^{8/3}\right).$$
\end{theorem}
\begin{proof}
Theorem 1 and Remark 15 from \cite{SZ14} give this statement for curves in $\C^2$.
We can reduce to that case using a generic projection argument, for instance as worked out in detail in \cite{PZ14}. 
We will only sketch how that argument can be adapted to this situation.

Let $\psi:\C^4\to\C^2$ be the projection $(z_1,z_2,z_3,z_4)\mapsto (z_1,z_3)$.
We claim that that there is a linear transformation $\varphi:\C^4\to \C^4$ 
with a matrix of the form
$$\begin{pmatrix}a&b&0&0\\c&d&0&0\\0&0&a'&b'\\0&0&c'&d'\end{pmatrix},$$
so that $\pi:=\psi\circ \varphi$ has the following properties: $\pi$ is bijective on $\Pi$; $\pi$ induces a bijection between $I(\Pi,\Gamma)$ and $I(\pi(\Pi), \pi(\Gamma))$; for $\gamma,\gamma'\in \Gamma$, $\pi(\gamma)$ and $\pi(\gamma')$ are distinct algebraic curves in $\C^2$.
Because of the form of the matrix, 
we can write $\pi(\Pi) = A'\times B'$ with two sets $A',B'\subset \C$.
The linear map does not increase the degree of the curves.
Applying the main theorem of \cite{SZ14} gives the desired bound.

The claim is proved exactly as in \cite[Corollary 2.5]{PZ14}, by showing that the set of $\varphi$ for which one of these properties fails is a lower-dimensional subset of the $8$-dimensional space of such matrices.
\end{proof}

By Lemma \ref{lem:intersections},
$\CP_1$ and $\CC_1$ almost exactly satisfy the conditions of Theorem \ref{thm:incidencebound} with $A =B= T^*$, $\mu = m$, $\delta = d^2$, and $\Delta=d^2$; only the condition $|\CC_1|=m^2$ need not quite hold, but it is easily forced by adding in dummy curves or points, without adding incidences.
Thus we get the following bound.

\begin{lemma}
We have the incidence bound
$$I(\CP_1,\CC_1)=O\left( d^{10/3}m^{8/3}\right).$$
\end{lemma}


\subsection{Conclusion}
We show that the incidences coming from $\CP_0$ and $\CC_0$ are negligible.

 \begin{lemma}
 If each of $C_1,C_2$ has $O(d^2)$ automorphisms, then
  $$|\CC_0|=O(d^2m)~~~\text{and}~~~|\CP_0|=O(d^2m).$$
 \end{lemma}
 \begin{proof}
 We define a graph with vertices $C_{ij}\in \CC_0$ and
 an edge between $C_{ij}$ and $C_{kl}$ if and only if $|C_{ij}\cap C_{kl}| = \infty$.
 We color an edge $C_{ij}C_{kl}$ with the transformation $T$ if $C_{ij}\cap C_{kl} = G_T$;
 by Lemma \ref{lem:infiniteintersection}, there is such a $T$ for every edge.
 
 If two edges of the form $C_{ij}C_{kl}$ and $C_{ij'}C_{k'l'}$ have the same color $T$,
 then $C_{ij}\cap C_{kl} = G_T = C_{ij'}\cap C_{k'l'}$.
 Then $G_T\subset C_{ij}\cap C_{ij'}$, so  $|C_{ij}\cap C_{ij'}| = \infty$,
 contradicting Lemma \ref{lem:infiniteintersection}.\footnote{In fact, the edges of the same color form a clique, 
but we do not need this fact.}
 
 It follows that every color $T$ occurs at most $m$ times, 
 since for each $i$ there is at most one $j$ 
 such that $C_{ij}$ is incident with an edge of color $T$.
 By assumption, $C$ has $O(d^2)$ automorphisms, so there are at most $O(d^2)$ colors,
 hence the graph has $O(d^2m)$ edges.
 By definition of $\CC_0$ there are no isolated vertices,
 so the number of vertices is at most twice the number of edges, 
 hence $|\CC_0|=O(d^2m)$.
  
 A similar argument applied to the dual curves gives the bound on $|\CP_0|$.
 \end{proof}

\begin{lemma}
 If each of $C_1,C_2$ has $O(d^2)$ automorphisms, then
 $$I(\CP,\CC_0) =O(d^3m^2)\hspace{20pt}\text{and}\hspace{20pt}
   I(\CP_0,\CC) =O(d^3m^2).$$
\end{lemma}
\begin{proof}
Any $C_{ij}$ has at most $dm$ incidences with points $(q_s,q_t)\in \CP$.
This is because for any of the $m$ choices for $q_s$,
the corresponding $q_t$ must be an intersection point of $C_2$ with the line
$\{q \in\C^2: B(p_j, q) = B(p_i,q_s)\}$.
Since we assumed that $C_2$ is not a line, by B\'ezout's inequality there are at most $d$ such intersection points.

Since $|\CC_0|=O(d^2m)$, this gives $I(\CP,\CC_0) =O(d^3m^2)$.
The dual argument gives the second bound.
\end{proof}

We get the overall incidence bound 
$$I(\CP,\CC) \leq I(\CP_0, \CC) + I(\CP,\CC_0) +I(\CP_1,\CC_1)
 =O\left(d^{10/3}m^{8/3}\right).$$
Combining this with $I(\CP,\CC) = |\CQ| \ge m^4/|\CB|$ from Lemma \ref{lem:lowerbound} and $m\geq n/d$ gives
$$|\CB(S,T)| = \Omega\left(|\CB|\right) 
= \Omega\left(m^4/|\CQ|\right) 
=\Omega\left(d^{-10/3} m^{4/3}\right)
=\Omega\left(d^{-14/3} n^{4/3}\right),$$
which completes the proof of Theorem \ref{thm:main2}.

 
\newpage
\section{Linear automorphisms}
\label{sec:autos} 
 
 In this section we study algebraic curves that have infinitely many linear automorphisms.
Although the topic seems classical, we were not able to find in the literature the exact statement that we need,
 so we provide our own proof.

 Recall that by a \emph{(linear) automorphism} of a curve $C$ we mean an invertible linear transformation $T:\C^2\to\C^2$ such that $T(C) = C$. 
 Note that in algebraic geometry, ``automorphism'' often denotes a \emph{polynomial} transformation (or ``morphism'') that fixes the curve,
 or sometimes a \emph{projective} transformation that fixes the curve.
 The classic theorem about polynomial automorphisms is Hurwitz's Theorem,
 which states that a nonsingular curve of genus $g\geq 2$ has at most $84(g-1)$ polynomial automorphisms 
 (see for instance \cite{Ha77}, Exercise IV.2.5). 
If $C$ has degree $d$, then we have $g\leq d^2$, so we get a bound in terms of the degree $d$.
However, this does not give the exact picture for linear automorphisms.
 For nonsingular curves, it would reduce the question to conics, 
 for which one can easily compute what the linear automorphisms are.
 However, there are many higher-degree singular curves of genus 0,
 for which it is harder to determine the linear automorphisms.
 This is what we do directly with an elementary approach,
 sidestepping Hurwitz's Theorem (and its difficult proof) altogether.

The theorem we prove in this section is the following.
Together with Corollary \ref{cor:main2}, it implies Theorem \ref{thm:main}.
Special curves are defined in Definition \ref{def:special}.

 \begin{theorem}
 \label{thm:autos}
  An irreducible algebraic curve of degree $d$ has $O(d^2)$ linear automorphisms,
  unless it is a special curve.
 \end{theorem}
\begin{proof}
The theorem follows directly from Lemma \ref{lem:notorigin} and Lemma \ref{lem:origin} below.
Assume $C$ is not a special curve.
If $C$ does not contain the origin, it has $O(d^2)$ automorphisms by Lemma \ref{lem:notorigin}, and if it does contain the origin, it has $O(d^2)$ automorphisms by Lemma \ref{lem:origin}.
\end{proof}

 \begin{example}
Special curves have infinitely many automorphisms.
 For $x^k=y^\ell$, the matrix
 $$ \begin{pmatrix}\alpha^\ell&0\\0&\alpha^k \end{pmatrix}$$
defines an automorphism for all $\alpha\in \C\backslash\{0\}$.
It then clearly follows that a linearly equivalent curve has infinitely many automorphisms.
\end{example}

An initial idea for proving Theorem \ref{thm:autos} would be to observe the following about an automorphism $T$ of the curve $C$.
If $L$ is an eigenline of $T$ and $q\in C\cap L$, then $T^i(q)\in C\cap L$ for all $i$.
If the eigenvalue of $L$ is not a root of unity, then the points $T^i(q)$ would form an infinite set in $C\cap L$, so by B\'ezout's inequality, $C$ would have to equal $L$. 
However, this approach fails, because $C\cap L$ may be empty (and this is indeed what happens for special curves).
We therefore have to use a similar but trickier argument.
Over $\R$, the argument would be considerably simpler, as we would not have to worry about roots of unity.

Our proof of Theorem \ref{thm:autos} rests on the three lemmas below. The first two are complementary and together imply Theorem \ref{thm:autos}.
The third, more technical, lemma is used in the proofs of the first two lemmas to handle specific subcases.
We use some concepts from the theory of algebraic curves, for which we refer to \cite{Ha77}; namely the projective plane, singularities and their branches, and intersection multiplicity.

In these lemmas we let $C$ be an irreducible algebraic curve of degree $d$, and $f$ a minimum-degree polynomial with $C = Z(f)$.
We write $T_\lambda$ for the scaling transformation defined by $T_\lambda(p)=\lambda p$, with $\lambda\in \C\backslash\{0\}$.
We write $L_m$ for the line $y = mx$ with $m\in \C$.

\begin{lemma}\label{lem:notorigin}
Suppose $C$ is not a line and does not contain the origin.
Then $C$ has $O(d^2)$ automorphisms, unless it is linearly equivalent to $x^ky^\ell=1$, with $k,\ell\geq 1$.
\end{lemma}
\begin{proof}
Suppose $C$ has more than $d^2$ automorphisms, and choose matrices $A_1,\ldots,A_{d^2+1}$ from among them.\footnote{We really mean $d^2$; the $O(d^2)$ in the lemma comes from the second part of this proof}
We claim that for all but finitely many $m\in \C$, the line $L_m$ has the following two properties: $|C\cap L_m|=d$,  and the lines $A_iL_m$ are distinct.
The first property fails only for the finitely many $m$ such that $L_m$ is tangent to $C$, or intersects $C$ at infinity or in a singularity.
The second property fails only when for some pair $i,j$, $L_m$ is a line such that $(A_i-A_j)L_m=0$; if such a line exists, it is unique.

Choose $L_m$ with the two properties above.
Suppose $q,\lambda q\in C\cap L_m$ for some $\lambda\in \C\backslash\{0,1\}$.
Then the points $A_iq, \lambda A_iq$ are all on $C$, and they are all distinct by the second property of $L_m$.
Since $T_\lambda$ sends $A_iq$ to $\lambda A_iq$, 
the irreducible curves $C$ and $T_\lambda(C)$ have $d^2+1$ points in common, so by B\'ezout's inequality we have $T_\lambda(C) = C$.
Thus $T_\lambda$ is an automorphism of $C$, and $T_\lambda^iq = \lambda^iq$ lies on $L_m\cap C$ for all $i\in \Z$.
If more than $d$ of the numbers $\lambda^i$ are distinct, then B\'ezout's inequality gives $C = L_m$.
Otherwise, $\lambda$ is a root of unity of order at most $d$.

Choose $q\in C\cap L_m$ and consider the argument in the previous paragraph for $q$ together with each of the $d-1$ other points in $C\cap L_m$ in the role of $\lambda q$. This, together with $\lambda =1$, gives $d$ distinct values of $\lambda$, each of which is a root of unity of order at most $d$.
This implies that one of these $\lambda$ is a primitive $d$-th root of unity, i.e., $\lambda^d=1$ but $\lambda^k\neq 1$ for $0<k<d$. 

Let $T_\lambda$ an automorphism of $C$ with $\lambda$ a primitive $d$-th root of unity.
Write $q=(q_x,q_y)$.
Then, for any $m$ as above,
$\lambda^i q_x$ must be a root of $f(x,mx)$ for each $i=0,\ldots,d-1$.
Thus
$$f(x,mx) = \alpha\prod_{i=0}^{d-1}(x - \lambda^iq_x) 
= \alpha (x^d-q_x^d)$$
for some $\alpha \in \C\backslash\{0\}$.
Because this holds for all but finitely many $m$,
it follows that $f(x,y)$ only has terms of degree $0$ or $d$, and (after scaling) there are $a_i,c\in \C$ such that
$$f(x,y) = \prod_{i=1}^d(y-a_ix) + c.$$ 

The lines $L_{a_i}$ are the asymptotes of $C$, and any automorphism of $C$ must permute these lines (i.e., it must permute the set $\{L_{a_i}\}$).
In Lemma \ref{lem:threeatinfinity} we will show that
$C$ has $O(d^2)$ automorphisms if it permutes a set of three or more lines, so we are done if at least three $a_i$ are distinct.
Otherwise, only two of the $a_i$ are distinct, which means that
$$f(x,y) = (y-b_1x)^k(y-b_2x)^\ell + c$$ 
for some integers $k,\ell$ and $b_1,b_2,c\in\C$.
This equation is linearly equivalent to $x^ky^\ell=1$.
\end{proof}

\begin{lemma}\label{lem:origin}
Suppose $C$ is not a line and contains the origin.
Then $C$ has $O(d^2)$ automorphisms, unless it is linearly equivalent to $x^k=y^\ell$,
with $k,\ell\geq 1$.
\end{lemma}
\begin{proof}
Now $C$ need not have exactly $d$ distinct points on most lines $L_m$, since if it has a singularity at the origin, it may have high intersection multiplicity with all lines at the origin.
However, there is a $k\leq d$ such that most lines have $|L_m\cap C|=k$.
By the same argument as in Lemma \ref{lem:notorigin}, we can reduce to the case where 
$T_\lambda$ is an automorphism, with $\lambda^k=1$,
and $C\cap L_m$ consists of the points $\lambda^iq$ for $i = 0,\ldots, k$.
Hence, for most $L_m$ we have
$$f(x,mx) = (\alpha x^k+\beta)x^{d-k} = \alpha x^d+\beta x^{d-k},$$
and it follows that
$$f(x,y) = a\prod_{i=1}^d(y-a_ix) 
+ b\prod_{j=1}^{d-k}(y-b_jx).$$ 

Any automorphism must permute the asymptotes $L_{a_i}$, and it must also permute the lines $L_{b_j}$, because these are the tangent lines of $C$ at the origin.
Note that the lines $L_{b_j}$ are distinct from the lines $L_{a_i}$ because $f$ is irreducible.
By Lemma \ref{lem:threeatinfinity}, if at least three of all these lines together are distinct, then $C$ has $O(d^2)$ automorphisms. Otherwise, we must have all $a_i$ equal and all $b_j$ equal, so 
$$f(x,y) = a(y-a'x)^d + b(y-b'x)^{d-k},$$
which is linearly equivalent to $x^d=y^{d-k}$.
\end{proof}

\begin{lemma}\label{lem:threeatinfinity}
Let $\mathcal{L}$ be a set of lines through the origin in $\C^2$, with $3\leq |\mathcal{L}|\leq 2d$. 
Then an algebraic curve $C\subset \C^2$ has $O(d^2)$ automorphisms that permute $\mathcal{L}$.
\end{lemma}
\begin{proof}
We work in the projective plane.
Let $L_\infty$ be the line at infinity and $P_\infty$ the set of points at infinity of the lines in $\mathcal{L}$,
so $3\leq |P_\infty|\leq 2d$.
For a linear $T$ on $\C^2$ we write $\varphi_T$ for the M\"obius transformation that $T$ induces on $L_\infty$. 
We note that any such M\"obius transformation is determined by its image on any three points.
Let $G$ be the group of automorphisms of $C$ that permute $\mathcal{L}$, and $G_\infty:=\{\varphi_T:T\in G\}$,
so every $\varphi\in G_\infty$ permutes $P_\infty$.

We first note that $G$ and $G_\infty$ are finite groups.
Since $|P_\infty|\geq 3$, a permutation of $P_\infty$ corresponds to at most one transformation in $G_\infty$, which implies that $G_\infty$ is finite.
To show that $G$ is finite, we show that for any $\varphi\in G_\infty$ there are finitely many $T\in G$ such that $\varphi_T=\varphi$.
Choose two points of $C$ in $\C^2$ that do not lie on the same line through the origin.
Then for a fixed $\varphi\in G_\infty$, any $T\in G$ with $\varphi_T=\varphi$ must send these two points to points on two fixed lines, and given the images of these two points, $T$ is determined. 
Since $C$ has at most $d$ points on these lines, there are finitely many possible images for these points, which implies that there are at most finitely many such $T$.
Thus $G$ is also finite.\footnote{This rough argument already gives a bound on $|G|$, but it is too large for our purposes.}

We now use some basic facts about M\"obius transformations, which can be found in for instance \cite[Chapter 3]{N97} or \cite[Chapter 2]{JS87}.
A M\"obius transformation of finite order has exactly two fixed points.
A finite subgroup of the group of M\"obius transformations is either a cyclic group, a dihedral group, or one of $S_4$, $A_4$, or $A_5$ (see \cite[Corollary 2.13.7]{JS87}), so $G_\infty$ must be one of these groups.
In the last three cases, $G_\infty$ has size at most $60$.
If $G_\infty$ is cyclic, then every $\varphi\in G_\infty$ has the same two fixed points.
Since $|P_\infty|\geq 3$, we can choose a $p\in P_\infty$ that is not one of the two fixed points,
and then choosing the image of $p$ from the $|P_\infty|\leq 2d$ candidates determines $\varphi$.
Thus $|G_\infty|\leq 2d$.
 If $G_\infty$ is dihedral, there are two points such that any $\varphi\in G_\infty$ either fixes them, or swaps them.
The same argument as for the cyclic case then gives that $|G_\infty|\leq 4d$.
Altogether we have $|G_\infty|\leq \max\{4d, 60\} = O(d)$.

Fix $\varphi\in G_\infty$ and choose a point $q\in C$ on a line $L$ through the origin that corresponds to a fixed point of $\varphi$. 
Then for any $T\in G$ with $\varphi_T=\varphi$, $T(q)$ must lie on a $L$, as well as on $C$. 
Since $C$ is not a line, it has at most $d$ points on $L$.
Thus there are at most $d$ choices for $T(q)$, and given this choice, $T$ is determined.
It follows that $|G|\leq d\cdot |G_\infty| = O(d^2)$.
\end{proof}


\section{Discussion}\label{sec:discussion}

\paragraph{Degree dependence.}
Let $F:\C^2\times\C^2\to\C$ be a polynomial function.
Given a set $S$ of $n$ points in $\C^2$, by interpolation there exists an algebraic curve of degree $O(n^{1/2})$ containing $S$.
Thus, a bound $\Omega(d^{-\alpha}n^{1+\beta})$ for the number of distinct values of $F$ on a curve gives a lower bound $|F(S)|=\Omega(n^{1+\beta-\alpha/2})$ on the number of distinct values of $F$.

In \cite{PZ14}, where $F(p,q) = (p_x-q_x)^2+(p_y-q_y)^2$ was the Euclidean distance function, the bound obtained (over $\R$) was $\Omega(d^{-11}n^{4/3})$, which clearly makes the interpolation argument above useless. Part of the goal for this paper was to see if this could be improved for bilinear forms. Over $\C$, our main bound from Theorem \ref{thm:main} also gives nothing.
Over $\R$, our proof would give $\Omega(d^{-2}n^{4/3})$ (mainly because the dependence on $d$ in the real equivalent of Theorem \ref{thm:incidencebound} would be better; see \cite{SZ14}).
Then interpolation gives $|F(S)|=\Omega(n^{1/3})$, which is more tangible but still rather weak.

We conclude that to obtain an interesting bound from this interpolation argument, one would have to improve the exponent $4/3$, or the dependence on $d$ in Theorem \ref{thm:incidencebound}.

\paragraph{Elekes-R\'onyai on curves.}
Our result fits into the general framework of Elekes and R\'onyai \cite{ER00},
which considers polynomial functions
$$F:X_1\times X_2 \to X_3,$$
for varieties $X_1,X_2,X_3$ of the same dimension.
Elekes and R\'onyai \cite{ER00} consider the case where $X_1=X_2=X_3=\R$, and proved that $F$ takes $\omega(n)$ values, unless it has one of the special forms $F(x,y) = G(H(x)+K(y))$ or $F(x,y)=G(H(x)\cdot K(y))$ for polynomials $G,H,K$.
The lower bound was improved by Raz, Sharir, and Solymosi \cite{RSS14} 
to $\Omega(n^{4/3})$.

In our case we have $X_1 = X_2 = C$ and $X_3 = \C$, and $F$ a bilinear polynomial.
We note that if $M$ is not invertible, we have $B(p,q) = L_1(p)\cdot L_2(q)$ for linear polynomials $L_1,L_2$, which one can see as an analog of the multiplicative form of Elekes and R\'onyai (an additive form is actually not possible here).
This (and other, unpublished, considerations) leads us to the following conjecture.

\begin{conjecture}
Let $C\subset \C^2$ be an algebraic curve of degree $d_C$ and $F:C\times C\to \C$ a polynomial of degree $d_F$.
Then for any $S\subset C$ we have
$$|F(S)|=\Omega_{d_C,d_F}(|S|^{4/3}),$$
unless $F(p,q) = G(H(p)+K(q))$, $F(p,q)=G(H(p)\cdot K(q))$,
or unless $C$ is rational.
\end{conjecture}

It seems reasonable to take rational curves as exceptions in this statement, because these are the curves that can have infinitely many automorphisms defined by higher-degree polynomials (essentially by Hurwitz's Theorem, see Section \ref{sec:autos}).
Of course, for specific functions the exact class of exceptions may be smaller.

When $F(p,q) = G(H(p)+K(q))$ or $F(p,q)=G(H(p)\cdot K(q))$, 
$|F(S_1,S_2)| = O(n)$ is possible for different sets $S_1, S_2$.
For the additive form, choose a set $S_1$ of intersection points of $C$ with the curve $H(p)=i$ for $i = 1,\ldots, |S|$, and a set $S_2$ of intersection points with $K(q)=j$ for $j= 1,\ldots, n$ (this is certainly possible over $\C$; over $\R$ one needs to be more careful).
For the multiplicative form, one can do the same with $H(p) = 2^i$ and $K(q)=2^j$.
However, it seems difficult to construct such an example with $S_1 = S_2$, unless $H=K$.

\paragraph{The exponent $4/3$.}
The exponent $4/3$ is not expected to be tight.
In all of the papers \cite{SSS13, PZ14, RSS14} that obtain it in this framework, the main open problem is to improve this exponent, perhaps as far as $\Omega(|S|^{2-\varepsilon})$.
In these proofs, the room for improvement seems to be in the incidence bound. 
Perhaps one can improve on the Szemer\'edi-Trotter-like exponent in Theorem \ref{thm:incidencebound} by using the specific nature of the incidence problem that one gets here, with the point set being a Cartesian product, and the curves being a very restricted family.
Indeed, the curves are dual to a point set that is also a Cartesian product.


\vspace{22pt}
\noindent{\bf\Large Acknowledgments.}\\
\vspace{-4pt}

\noindent Both authors were partially supported by Swiss National Science Foundation Grants 200020-144531 and 200021-137574.  Part of this research was performed during the second author's visit to the Institute for Pure and Applied Mathematics in Los Angeles, which is supported by the National Science Foundation. 
The authors thank J\'anos Pach for all his support.

\end{document}